\documentclass[a4paper,10pt]{amsart}

\usepackage{amsmath}%
\usepackage{amsfonts}%
\usepackage{amssymb}%
\usepackage{dsfont}%
\usepackage{graphicx}%
\usepackage{xspace}%
\usepackage[english]{babel}%
\usepackage[babel]{csquotes}%
\usepackage[T1]{fontenc}%
\usepackage{a4wide}%
\usepackage{bm}

\newtheorem{theorem}{Theorem}[section]

\newtheorem{proposition}[theorem]{Proposition}
\newtheorem{prop}[theorem]{Proposition}

\theoremstyle{definition}
\newtheorem{definition}[theorem]{Definition}

\newtheorem{remark}[theorem]{Remark}

\newtheorem{example}[theorem]{Example}

\DeclareMathOperator{\IR}{\mathds{R}}
\DeclareMathOperator{\IZ}{\mathds{Z}}
\DeclareMathOperator{\IC}{\mathds{C}}
\DeclareMathOperator{\R}{\mathds{R}}
\DeclareMathOperator{\Z}{\mathds{Z}}
\DeclareMathOperator{\C}{\mathds{C}}

\DeclareMathOperator{\conv}{\mathrm{conv}}
\DeclareMathOperator{\ohne}{\smallsetminus}
\DeclareMathOperator{\leer}{\varnothing}
\DeclareMathOperator{\Par}{\mathrm{par}}
\DeclareMathOperator{\Span}{\mathrm{span}}

\def\floor#1{\left\lfloor #1 \right\rfloor}

\def\fractional#1{\left\{ #1 \right\}}

\def\ie{i.\,e.\xspace}
\def\eg{e.\,g.\xspace}

\begin{document}
\selectlanguage{english}

\title{A generalization of a theorem of G. K. White}
\author[Victor Batyrev]{Victor Batyrev}
\address{Mathematisches Institut, Universit\"at T\"ubingen, 
Auf der Morgenstelle 10, 72076 T\"ubingen, Germany}
\email{victor.batyrev@uni-tuebingen.de}
\author[Johannes Hofscheier]{Johannes Hofscheier} \address{Mathematisches Institut, Universit\"at T\"ubingen, Auf der Morgenstelle 10, 72076 T\"ubingen, Germany}
\email{johannes.hofscheier@student.uni-tuebingen.de}

\begin{abstract}
An $n$--dimensional simplex $\Delta$ in $\R^n$ is called \emph{empty lattice simplex} if $\Delta\cap\Z^n$ is exactly the set of vertices of $\Delta$. A theorem of White shows that if $n=3$ then any empty lattice simplex $\Delta\subset\R^3$ is isomorphic up to an unimodular affine linear transformation to a lattice tetrahedron whose all vertices have third coordinate $0$ or $1$. In this paper we prove a generalization of this theorem for an arbitrary odd dimension $n=2d-1$ which in some form was conjectured by Seb\H{o} and Borisov. This result implies a classification of all $2d$--dimensional isolated Gorenstein cyclic quotient singularities with minimal $\log$--discrepancy at least $d$.
\end{abstract}
\maketitle
\section{Introduction}
\label{sec:Introduction}

We will be working on the $d$--dimensional real space $\R^d$ equipped with the standard lattice $\Z^d$ and we will denote the standard basis of $\Z^d$ by $e_1,\ldots,e_d$. Let $v_1,\ldots,v_{k+1}$ be affinely independent vectors of $\R^d$. Then the convex set generated by $v_1,\ldots,v_{k+1}$ is a \emph{$k$--dimensional simplex}. If the vectors $v_1,\ldots,v_{k+1}$, which form the vertices of $\Delta$, are contained in $\Z^d$ then we call $\Delta$ a \emph{lattice simplex}. For the following let $\Delta$ be a lattice simplex. If the only lattice points which are contained in $\Delta$, are its vertices then we call $\Delta$ an \emph{empty} lattice simplex. In \cite{White:LatticeTetrahedra} G. K. White poses the general problem to investigate the properties of empty lattice simplices and, if possible, classify them. Of course by \enquote{classify} we mean a classification up to a suitable notion of isomorphism, namely up to affine linear isomorphisms which respect the lattice $\Z^d$. Many mathematicians already worked directly or indirectly on this question among these are \cite{White:LatticeTetrahedra}, \cite{Sebo:EmtpyLatticeSimplices}, \cite{Borisov:QuotSing}. One paper which is of particular interest for us is \cite{White:LatticeTetrahedra}. There G. K. White gives a full classification of empty lattice simplices if the dimension of the ambient space is fixed to $3$.

\begin{theorem}[White]\label{thm:OldWhite}
 Let $\Delta$ be a $3$--dimensional lattice simplex (also called a \emph{lattice tetrahedron}) of $\R^3$. Then the following statements are equivalent
 \begin{enumerate}
  \item $\Delta$ is empty
  \item $\Delta$ is affine unimodular isomorphic to a lattice simplex $\conv(v_1,v_2,v_3,v_4)\subset\R^3$ such that the last coordinate of $v_1,v_2$ is 0 and of $v_3,v_4$ is 1 and such that the edges $\Delta_1=\conv(v_1,v_2),\Delta_2=\conv(v_3,v_4)$ are empty.
 \end{enumerate}
\end{theorem}

\begin{figure}[ht]
 \includegraphics{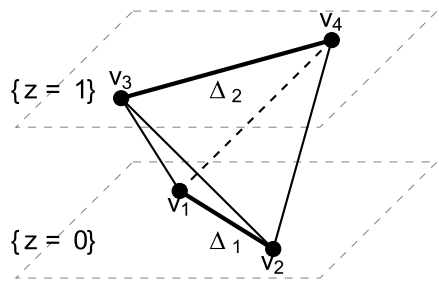}
 \label{fig:Illustration theorem of White}
 \caption{Illustration to the theorem of White.}
\end{figure}

The aim of this paper is to generalize this theorem to arbitrary odd dimensions. We proceed in two steps:
\begin{enumerate}
 \item We need to find a suitable replacement for emptyness which coincide in dimension $3$. As one might guess \enquote{just} emptyness is a far too weak notion or in other words: the family of empty lattice simplices seems to be too large to give an easy classification for.
 \item We need to find a generalization of the construction used in the theorem of White.
\end{enumerate}

For the first problem we will give two equivalent answers, one of which will be in elementary terms while the other one will use the $h^*$--polynomial (see definition below). The second problem will be solved by introducing the notion of the Cayley polytope (see \cite{BatyrevNill:CombAspectsMirrorSymmetry} or \cite{BatyrevNill:LatticePolytopes}).

\begin{definition}[Cayley Polytope]
 Consider $r$ ($r\le d$) lattice polytopes $\Sigma_1,\ldots,\Sigma_r\subset\R^d$ and the cone $\sigma:=\{(\lambda_1,\ldots,\lambda_r,\sum\lambda_i\Delta_i)\subset\R^{d+r}:\lambda_i\ge 0\}$. Then the polytope which arises by intersecting the cone $\sigma$ with the hyperplane $H:=\{(x_1,\ldots,x_{d+r})\in\R^{d+r}:\sum_{i=1}^{r}x_i=1\}$ is called the \emph{Cayley polytope} of $\Sigma_1,\ldots,\Sigma_r$ and is usually denoted by $\Sigma_1*\ldots*\Sigma_r$. It obviously is a lattice polytope.

In other words: $\Sigma_1*\ldots*\Sigma_r$ is the convex hull of the polytopes $e_1\times\Sigma_1,\ldots,e_r\times\Sigma_r$ in $\R^{d+r}$.
\end{definition}

\begin{remark}
We want to mention that the assertion of the second part of theorem \ref{thm:OldWhite} is equivalent to saying that $\Delta$ is isomorphic to the Cayley polytope generated by the empty lattice segments $\Delta_1:=\conv(v_1,v_2)\subset\IR^2$ and $\Delta_2:=\conv(v_3,v_4)\subset\IR^2$, \ie $\Delta\cong\Delta_1*\Delta_2$ (for details see the next section).
\end{remark}

Furthermore we need the following definition.

\begin{definition}
 A $k$--dimensional lattice simplex $\Delta=\conv(v_1,\ldots,v_{k+1})\subset\R^d$ is called a \emph{basic lattice simplex} if one of the following two equivalent conditions are satisfied:
 \begin{enumerate}
  \item There is an affine unimodular isomorphism $\varphi:\R^d\rightarrow\R^d$ such that
   \[
    \varphi(\Delta)=\conv(0,e_1,\ldots,e_k)
   \]
  \item $v_1-v_{k+1},v_2-v_{k+1},\ldots,v_k-v_{k+1}$ is part of a lattice basis for $\Z^d$.
 \end{enumerate}
\end{definition}

\begin{remark}
We want to mention that a $3$--dimensional simplex $\Delta\subset\IR^3$ being empty implies that all its codimension $1$ facets are basic (for details see the next section).
\end{remark}

With all these notions at hand we can state our generalization of the theorem of White in the following form:

\begin{theorem}\label{thm:GeneralizedWhite}
 Let $\Delta=\conv(v_1,\ldots,v_{2d})\subset\R^{2d-1}$ be a $(2d-1)$--dimensional lattice simplex whose codimension $1$ facets are basic lattice simplices. Then the following statements are equivalent:
 \begin{enumerate}
  \item For all $k=1,\ldots,d-1$ there are no lattice points in the interior of $k\Delta$, \ie $\forall k=1,\ldots,d-1$: $\mathrm{Int}(k\Delta)\cap\Z^{2d-1}=\leer$.
  \item $\Delta\cong\Delta_1*\ldots*\Delta_d$ for $1$--dimensional lattice simplices $\Delta_i\subset\R^d$.
 \end{enumerate}
\end{theorem}

\begin{remark}
 We want to remark that Seb\H{o} somehow expected in \cite[Conjecture 4.1]{Sebo:EmtpyLatticeSimplices} our generalization of the theorem of White.
\end{remark}

The proof of the generalization of the theorem of White will be based on a result of Morrison and Stevens concerning the \emph{\textit{1}st (periodic) Bernoulli function} which maps a real number $x$ to
\[
 B_1(x)=\begin{cases}
         \fractional{x}-\frac{1}{2}&,x\not\in\IZ\\
         0&,x\in\IZ
        \end{cases}
\]
where $\fractional{x}$ denotes the fractional part of $x$, \ie $\fractional{x}=x-\floor{x}$ where $\floor{x}$ is the biggest integer which is smaller than or equal to $x$. This function naturally arises in the taylor expansion of the function
\[
 \frac{te^{xt}}{e^t-1}=\sum_{n=0}^\infty \tilde B_n(x)\frac{t^n}{n!}
\]
So $\tilde B_1(x)=x-\tfrac{1}{2}$. By making this function periodic with period $1$ by inserting the fractional part of $x$ instead of $x$ and letting $B_1(0)=B_1(1)=0$ one gets our previous definition.

In \cite[Section 1, Corollary 1.3]{MorrisonStevens:TerminalQuotSing} Morrison and Stevens proved the following proposition. Actually they only showed the assertion for $d=4$ but it can be easily generalized to arbitrary $d$.
\begin{proposition}\label{prop:BernoulliFunctions}
 Let $d,n$ be positive integers. Let $a_1,\ldots,a_d$ be integers relatively prime to $n$. If for all $t\in\IZ$
  \[
   \sum_{i=1}^dB_1\left(\frac{ta_i}{n}\right)=0
  \]
 then the integer $d$ is even and after reordering the $a_i$ we have that $a_i+a_{i+1}\equiv0\pmod{n}$ for all $i=1,3,5,\ldots,d-1$.
\end{proposition}
Since Morrison and Stevens did not prove this version we will give a complete proof in section \ref{sec:BernoulliFunctions}.
\begin{remark}
 We want to remark that proposition \ref{prop:BernoulliFunctions} is a special case of a conjecture of Borisov \cite[Conjecture 2]{Borisov:QuotSing}.
\end{remark}

Furthermore we want to mention that this result gives a full classification of isolated $2d$--dimensional cyclic quotient singularities with minimal $\log$--discrepancy greater or equal to $d$.

A cyclic quotient singularity of dimension $d$ is the affine variety obtained by taking the quotient of $\C^d$ by a linear action of the group $\mu_n$ of $n$-th roots of unity. By diagonalizing we can assume that the linear actions is given by
\[
 \mu_n\times\C^d\rightarrow\C^d;\left(\zeta,(x_1,\ldots,x_d)\right)\mapsto\left(\zeta^{a_1}x_1,\ldots,\zeta^{a_d}x_d\right)
\]
for integers $a_i$. We call this quotient singularity of type $(a_1/n,\ldots,a_d/n)$. If $a_i=0$ for some $i=1,\ldots,d$ then the quotient singularity equals $X'\times\C$ for a lower dimensional quotient singularity $X'$. Hence we may assume that $a_i\neq0$ for all $i$. The quotient singularity $\C^d/\mu_n$ of type $(a_1,/n,\ldots,a_d/n)$ has an isolated singularity at the origin if and only if $\gcd(a_i,n)=1$ for all $i=1,\ldots,d$ (see \cite[Corollary 2.2]{MorrisonStevens:TerminalQuotSing}). To a cyclic quotient singularity one associates the minimal $\log$--discrepancy which is a birational invariant. In the cyclic case it has the following simple combinatorial description (see \cite{Borisov:MinDiscrToricSing} or \cite{Reid:DecompoToricMorph}).
\begin{definition}
 Let $\C^d/\mu_n$ be an isolated quotient singularity of type $(a_1/n,\ldots,a_d/n)$. Then the \emph{minimal $\log$--discrepancy} is given by
 \[
  \min\{\sum_{i=1}^d\fractional{\frac{ta_i}{n}}:t=1,\ldots,n-1\}
 \]
\end{definition}
Furthermore a quotient singularity $\C^d/\mu_n$ is called \emph{Gorenstein} if the image of the morphism $\mu_n\rightarrow\mathrm{GL}(n;\C)$ induced by the linear action of $\mu_n$ on $\C^d$ is contained in $\mathrm{SL}(n;\C)$. This again has a purely combinatorial description.
\begin{prop}
 Let $X=\C^d/\mu_n$ be a cyclic quotient singularity of type $(a_1/n,\ldots,a_d/n)$. Then $X$ is Gorenstein if and only if
 \[
  \sum_i\frac{a_i}{n}\in\Z
 \]
\end{prop}
Then we can prove the following theorem.
\begin{theorem}\label{thm: Cyclic Quotient Singularities}
 Let $\C^{2d}/\mu_n$ be an isolated Gorenstein cyclic quotient singularity of the type $(a_1/n,\ldots,a_{2d}/n)$. Then the following two statements are equivalent:
 \begin{enumerate}
  \item the minimal $\log$--discrepancy of $(a_1/n,\ldots,a_{2d}/n)\ge d$
  \item after reordering the $a_i$, it is $a_{2i-1}=-a_{2i}$ for all $i=1,\ldots,d$
 \end{enumerate}
\end{theorem}
This is an easy consequence of proposition \ref{prop:BernoulliFunctions}.
\begin{remark}
 In theorem \ref{thm: Cyclic Quotient Singularities} we can omit the assumption that the cyclic quotient singularity has to be Gorenstein since minimal $\log$--discrepancy$\ge d$ and $a_{2i-1}=-a_{2i}$ for all $i=1,\ldots,d$ already imply that $X$ is Gorenstein.
\end{remark}
\section{Lattice simplices}
\label{sec:LatticeSimplices}

In this section we are going to thoroughly introduce and discuss the notions we sketched in the introduction.

Let us begin with illustrating why the Cayley polytope generalizes the construction given in the theorem of White. We start by reformulating this construction: Let $\Delta=\conv(v_1,v_2,v_3,v_4)\subset\R^3$ be a $3$--dimensional lattice simplex such that $v_1,v_2\in\{z=0\}$ and $v_3,v_4\in\{z=1\}$ and $\conv(v_1,v_2),\conv(v_3,v_4)\subset\R^3$ are empty $1$--dimensional lattice simplices. Since $\conv(v_1,v_2)\subset\{z=0\}\subset\R^3$ is empty, there exists a point $w\in\{z=0\}$ such that $\{v_2-v_1,w\}$ is a basis for $\{z=0\}\cap\Z^3$ (see  \cite[Corollary 21.2]{Gruber:ConvexGeometry}). Since the third coordinate of $v_3-v_1$ equals $1$, $\{v_2-v_1,w,v_3-v_1\}$ is a lattice basis for $\Z^3$. Now we consider the lattice simplex $\Delta':=\conv((v_1,1),\ldots,(v_4,1))\subset\R^4$. Obviously $\{(v_2,1)-(v_1,1),(w,0),(v_3,1)-(v_1,1),(v_1,1)\}$ is a lattice basis for $\Z^4$. Let $\varphi:\R^4\rightarrow\R^4$ be the unimodular linear isomorphism which maps this basis to $\{0\times e_1,0\times e_2,(e_2-e_1)\times 0, e_1\times 0\}$. Since the third coordinate of $v_4-v_1$ is $1$ we get that $v_4-v_1=a(v_2-v_1)+nw+(v_3-v_1)$ for integers $a,n$. Then $(v_4,1)=a(v_2-v_1,0)+n(w,0)+(v_3-v_1,0)+(v_1,1)$, \ie $\varphi(v_4,1)=e_2\times(a,n)$. Thus $\Delta\cong\varphi(\Delta')=\Delta_1*\Delta_2$ is a Cayley polytope for $\Delta_1:=\conv(e_1\times 0,e_1\times e_1),\Delta_2:=\conv(e_2\times 0, e_2\times(a,n))\subset\R^2$.

If $n$ is a negative integer, then we can take $(-w,0)$ instead of $(w,0)$. So without any loss of generality we can assume that $n$ is non negative. Since the empty lattice simplex $\Delta$ is $3$--dimensional, the integer $n$ must be nonzero. Furthermore  either $a=0,n=1$ or the integer $a$ is coprime to the integer $n$. Let us assume by contradiction that there exists a common divisor $k>1$ of $a$ and $n$. Then the point
\[
 \frac{k-1}{k}v_3+\frac{1}{k}v_4\in\Delta_2\cap\Z^3
\]
is not a vertex of $\Delta$ but contained in the lattice $\Z^3$. This contradicts the emptyness of $\Delta$.

So we have seen that every $3$--dimensional empty lattice simplex $\Delta$ is isomorphic to a Cayley polytope $\Delta_1*\Delta_2$ where $\Delta_1=\conv(e_1\times0,e_1\times e_1)$ and $\Delta=\conv(e_2\times0,e_2\times(a,n))$ for integers $a,n$ with $n>0$ satisfying $(1)_{\Z}=(a,n)_{\Z}$ where $(a,n)_{\Z}$ ($(1)_{\Z}$) denotes the ideal in $\Z$ generated by $a,n\in\Z$ ($1\in\Z$). Conversely any Cayley polytope $\Sigma_1*\Sigma_2$ for two linear independent $1$--dimensional empty lattice simplices $\Sigma_i\subset\R^2$ is isomorphic to $\conv(0\times\Sigma_1,e_1\times\Sigma_2)$, \ie is of the type described in the theorem of White. Hence we can restate the theorem of White as follows.

\begin{theorem}\label{thm:White new version with Cayley polytope}
 Let $\Delta$ be a $3$--dimensional lattice simplex of $\Z^3\subset\R^3$. Then the following statements are equivalent:
 \begin{enumerate}
  \item $\Delta$ is empty
  \item $\Delta$ is isomorphic to a Cayley polytope $\Sigma_1*\Sigma_2$ of two affine linear independent empty lattice segments $\Sigma_1,\Sigma_2\subset\R^2$.
 \end{enumerate}
\end{theorem}

To be more precisely we have shown that any $3$--dimensional empty lattice simplex $\Delta\subset\R^3$ is isomorphic to one of the lattice simplices of the following more general example.

\begin{example}\label{exm:representativeOfIsomorphismClasses}
 Consider the empty lattice polytopes $\Delta_i:=\conv(0,e_i)\subset\R^d$ for $i=1,\ldots,d-1$ and $\Delta_d:=\conv(0,(a_1,\ldots,a_{d-1},n))\subset\R^d$ where $a_1,\ldots,a_{d-1},n$ are integers with $n>0$ such that $(a_1\cdots a_{d-1},n)_{\Z}=(1)_{\Z}$. Then we will be interested in the Cayley polytope $\Delta_1*\ldots*\Delta_d$ of the affine lattice plane $H:=\{(x_1,\ldots,x_{2d})\in\R^{2d}:\sum_{i=1}^dx_i=1\}$ which we want to denote by $\Delta(a_1,\ldots,a_{d-1},n)$.
\end{example}

We will show that any lattice simplex of our generalization of the theorem of White is isomorphic to one of the lattice simplices $\Delta(a_1,\ldots,a_{d-1},n)$ of the previous example.

Next we introduce the following algebraic quantity one associates to a lattice simplex (see \cite{BeckRobins:ComputingTheContinuousDiscretely}).

\begin{definition}[$h^*$--polynomial]
 Suppose $\Delta$ is a $d$--dimensional lattice simplex of $\Z^d$ with vertices $v_1,v_2,\ldots,v_{d+1}$, and let $w_j=(v_j,1)$. We denote by $\left|k\Delta\cap\Z^d\right|$ the number of lattice points contained in the $k$\textit{th} multiple of $\Delta$. Then the Ehrhart--series is a rational function
 \[
  1+\sum_{k\ge1}\left|k\Delta\cap\Z^d\right|t^k=\frac{h_0^*+h_1^*t+\ldots+h_d^*t^d}{(1-t)^{d+1}}
 \]
 where $h_k^*$ equals the number of lattice points in the \emph{parallelepiped}
 \[
  \Par(\Delta):=\{\lambda_1w_1+\lambda_2w_2+\ldots+\lambda_{d+1}w_{d+1}:0\le\lambda_1,\lambda_2,\ldots,\lambda_{d+1}<1\}
 \]
 with the last coordinate equal to $k$. The polynomial $\sum h_i^*t^i$ we call the \emph{$h^*$--polynomial} of $\Delta$ and denote it by $h_\Delta^*$.
\end{definition}

Let us determine the $h^*$--polynomial of the lattice simplices $\Delta$ appearing in the theorem of White. In order to determine the coefficients $h_i^*$ of $h_\Delta^*$ we can proceed as follows: Put the simplex $\Delta$ on the hyperplane $\{x_4=1\}\subset\Z^4$, \ie consider $\widetilde{\Delta}:=\conv((v_1,1),\ldots,(v_4,1))$ where $v_1,\ldots,v_4$ are the vertices of $\Delta$. Then the number of lattice points lying in the intersection of the hyperplane $\{x_4=k\}\subset\Z^4$ and the parallelepiped $\Par(\Delta)$ is equal to the $k$\textit{th} coefficient $h_k^*$ of $h_\Delta^*$ (see figure \ref{Illustration Parallelepiped h-Polynomial}).

\begin{figure}[ht]
 \includegraphics{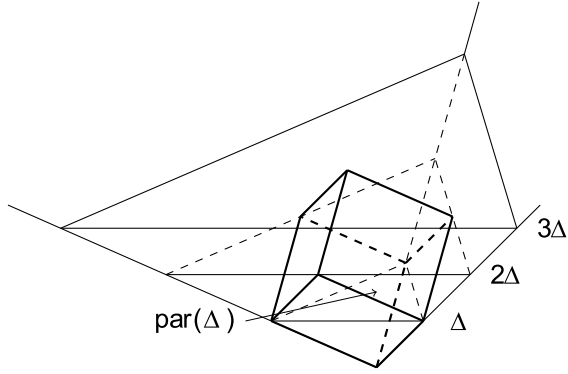}
 \label{Illustration Parallelepiped h-Polynomial}
 \caption{Illustration to the $h^*$--polynomial to a lattice simplex $\Delta$.}
\end{figure}

Thus $h_0^*=1$ and $h_1^*=0$ ($\Delta$ is empty) and $h_2^*$ is a non negative integer. In order to determine $h_3^*$ we make the following important observation: the intersection of $\Par(\Delta)$ with the hyperplane $\{x_4=3\}$ is empty. Assume by contradiction that there exists such a point, say $\lambda_1(v_1,1)+\ldots+\lambda_4(v_4,1)$. Then
\[
 (1-\lambda_1)(v_1,1)+\ldots+(1-\lambda_4)(v_4,1)
\]
is a lattice point of $\Par(\Delta)$ lying on the hyperplane $\{x_4=1\}$. This contradicts the emptyness of $\Delta$. Thus $h_3^*=0$. We have computed that $h_\Delta^*=1+h_2^*t^2$. This leads us to the following proposition.

\begin{prop}\label{prop: h-Polynomial vs lattice free simplex}
 Let $\Delta\subset\R^{2d-1}$ be a $(2d-1)$--dimensional lattice simplex. Then the following assertions are equivalent:
 \begin{enumerate}
  \item $\forall k=1,\ldots,d-1$, $k\Delta\cap\Z^{2d-1}\subset\Z v_1+\ldots+\Z v_{2d}$
  \item $\forall k=1,\ldots,d-1$, $\mathrm{Int}(k\Delta)\cap\Z^{2d-1}=\leer$ and all codimension 1 facets of $\Delta$ are basic lattice simplices
  \item $h^*_\Delta=1+(n-1)t^d$ for an integer $n>0$.
 \end{enumerate}
\end{prop}
\begin{proof}
 Let $\Delta=\conv(v_1,\ldots,v_{2d})\subset\R^{2d-1}$ be a $(2d-1)$--dimensional lattice simplex such that for all $k=1,\ldots,d-1$, $k\Delta\cap\Z^{2d-1}\subset\Z v_1+\ldots+\Z v_{2d}$. We claim that $\{\sum_{i\neq j}\lambda_i v_i:0\le\lambda_i<1\}\cap\Z^{2d-1}=\{0\}$ for all $j=1,\ldots,{2d}$. Assume that $\sum_{i\neq j}\lambda_i v_i\in\Z^{2d-1}$ for $0\le \lambda_i<1$. Then by assumption $\sum_{i\neq j}\lambda_i=0$ or $\sum_{i\neq j}\lambda_i=k\ge d$. The second case is not possible since
 \begin{align*}
  &\sum_{i\neq j}\lambda_i v_i\in\Z^{2d-1}\text{ with }\sum_{i\neq j}\lambda_i\ge d\\
  \Rightarrow&\sum_{i\neq j}\{1-\lambda_i\}v_i\in\Z^{2d-1}\text{ with }\sum_{i\neq j}\{1-\lambda_i\}\le 2d-1-\sum_{i\neq j}\lambda_i<d
 \end{align*}
 Thus the assertion is proved. By \cite[Corollary 21.2]{Gruber:ConvexGeometry} it follows that the facet of $\Delta$ which is generated by $v_1,\ldots,v_{j-1},v_{j+1},\ldots,v_{2d}$ for $j=1,\ldots,2d$ is a basic lattice simplex. Since $\mathrm{Int}(k\Delta)=\{\sum_i\lambda_iv_i:0<\lambda_i,\sum_i\lambda_i=k\}$ it obviously follows by assumption that $\mathrm{Int}(k\Delta)\cap\Z^{2d-1}=\leer$ for all $k=1,\ldots,d-1$.

 Now let $\Delta=\conv(v_1,\ldots,v_{2d})\subset\R^{2d-1}$ be a $(2d-1)$--dimensional lattice simplex whose comdimension $1$ facets are basic lattice simplices such that $\mathrm{Int}(k\Delta)\cap\Z^{2d-1}=\leer$ for all $k=1,\ldots,d-1$. We consider $\Par(\Delta)=\{\sum_i\lambda_i(v_i,1):0\le\lambda_i<1\}$. Then the k\textit{th} coefficient of $h_\Delta^*$ equals to the number of lattice points contained in $\Par(\Delta)\cap\{x_{2d}=k\}$. Let $\sum_i\lambda_i(v_i,1)\in\Par(\Delta)\cap\{x_{2d}=k\}$, \ie $0\le \lambda_i<1$ with $\sum_i\lambda_i=k$. Since the codimension $1$ facets of $\Delta$ are basic lattice simplices, it follows that $\lambda_i=0$ for all $i$ or $\lambda_i>0$ for all $i$. Let $k>0$. then all $\lambda_i>0$ for all $i$, \ie $\sum_i\lambda_i\in\mathrm{Int}(k\Delta)$. Hence $k\ge d$. We claim that $k\le d$ as well. Indeed assume that $k>d$. Then $\sum_i\{1-\lambda_i\}v_i\in\Z^{2d-1}$ with $\sum_i\{1-\lambda_i\}\le 2d-\sum_i\lambda_i<d$ contradicting the assumption. Hence we have seen that all coefficients of $h_\Delta^*$ are $0$ except for the $0$\textit{th} and the d\textit{th}, \ie $h_\Delta^*=1+(n-1)t^d$ for an integer $n>0$.

Finally let $\Delta=\conv(v_1,\ldots,v_{2d})\subset\R^{2d-1}$ be a $(2d-1)$--dimensional lattice simplex with $h_\Delta^*=1+(n-1)t^d$ for an integer $n>0$, \ie $\Par(\Delta)\cap\{x_{2d}=k\}=\leer$ for all $k=1,\ldots,d-1$. Let $\sum_i\lambda_iv_i\in\Z^{2d-1}$ for $0\le\lambda_i$ with $\sum_i\lambda_i=k\in\{1,\ldots,d-1\}$. Then $\sum_i\{\lambda_i\}v_i\in\Z^{2d-1}$ with $0\le\{\lambda_i\}<1$ and $0\le l:=\sum_i\{\lambda_i\}\le\sum_i\lambda_i=k\le d-1$, \ie $\sum_i\{\lambda_i\}v_i\in\Par(\Delta)\cap\{x_{2d}=l\}$ for an integer $0\le l\le d-1$. The case $l>0$ is not possible by assumption. Thus $l=0$. Then $\{\lambda_i\}=0$ for all $i$, \ie $\lambda_i\in\Z$ for all $i$.
\end{proof}

Obviously theorem \ref{thm:GeneralizedWhite} and proposition \ref{prop: h-Polynomial vs lattice free simplex} imply the following version of a generalization of theorem \ref{thm:OldWhite}.

\begin{theorem}
 Let $\Delta\subset\R^{2d-1}$ be a $(2d-1)$--dimensional lattice simplex whose codimension $1$ facets are basic lattice simplices. Then the following statements are equivalent:
 \begin{enumerate}
  \item $\forall k=1,\ldots,d-1$, $k\Delta\cap\Z^{2d-1}\subset\Z v_1+\ldots+\Z v_{2d}$.
  \item $\forall k=1,\ldots,d-1$, $\mathrm{Int}(k\Delta)\cap\Z^{2d-1}=\leer$
  \item $h_\Delta^*=1+(n-1)t^d$ for an integer $n\ge1$.
  \item $\Delta$ is isomorphic to one of the lattice simplices $\Delta(a_1,\ldots,a_{d-1},n)$ of example \ref{exm:representativeOfIsomorphismClasses}.
  \item $\Delta\cong\Delta_1*\ldots*\Delta_d$ for $1$--dimensional lattice simplices $\Delta_i\subset\R^d$.
 \end{enumerate}
\end{theorem}

 As an illustration consider the case $d=3$. Here the cases (1) and (2) become the single case (1').

\begin{theorem}
 Let $\Delta\subset\R^5$ be a $5$--dimensional lattice simplex. Then the following statements are equivalent:
 \begin{enumerate}
  \item[(1')] $\Delta$ is an empty lattice simplex and the only lattice points contained in $2\Delta$ are either its vertices or the midpoints of its edges.
  \item[(3)] $h_\Delta^*=1+(n-1)t^3$ for an integer $n\ge1$.
  \item[(4)] $\Delta\cong\Delta_1*\Delta_2*\Delta_3$ for $1$--dimensional lattice simplices $\Delta_1,\Delta_2,\Delta_3\subset\R^3$.
  \item[(5)] $\Delta$ is isomorphic to one of the lattice simplices $\Delta(a_1,a_2,n)$ of example \ref{exm:representativeOfIsomorphismClasses}.
 \end{enumerate}
\end{theorem}

Finally we want to give an interesting example which shows that none of the assumptions made in our generalization of the Theorem of White can be dropped, \eg one cannot omit the assumption that all codimension $1$ facets are basic lattice simplices.

\begin{example}\label{exam:Counterexample}
 For two integers $p,q$ we define the following linearly independent segments of $\R^3$

 \[
  \Delta_1:=\conv(0, (1, 0, 0))
  \Delta_2:=\conv(0, (1, p, 0))
  \Delta_3:=\conv(0, (1, 0, q))
 \]

 Then the Cayley polytope $\Delta_1*\Delta_2*\Delta_3$ is a $5$--dimensional lattice simplex of the affine lattice plane $\{x_1+x_2+x_3=1\}\subset\R^6$ and isomorphic to
 \[
  \Delta=\conv(0\times\Delta_1,e_1\times\Delta_2,e_2\times\Delta_3)\subset\R^5
 \]
 Then $\Delta$ is empty and the multiple $2\Delta$ contains lattice points which are not an integer linear combination of the vertices of $\Delta$, namely for $k=1,\ldots,p-1$ and $l=1,\ldots,q-1$
 \begin{align*}
  \frac{k}{p}(1,0,1,p,0)+\frac{p-k}{p}(1,0,0,0,0)+\frac{p-k}{p}(0,0,1,0,0)+\frac{k}{p}(0,0,0,0,0)=\\
  (1,0,1,k,0)\\
  \text{and }\frac{l}{q}(0,1,1,0,q)+\frac{p-l}{p}(0,1,0,0,0)+\frac{p-l}{p}(0,0,1,0,0)+\frac{l}{p}(0,0,0,0,0)=\\
  (0,1,1,0,l)\\
 \end{align*}
 To be more precisely this shows that $G(\Delta):=\Z^{6}/\left(\Z(v_1,1)+\ldots+\Z(v_{6},1)\right)\cong\Z/p\Z\times\Z/q\Z$ where $v_i$ are the vertices of $\Delta$ defined above. Indeed every element of $G(\Delta)$ can be uniquely represented by an element of $\Par(\Delta)$ and one of the sub lattice $\Z(v_1,1)+\ldots+\Z(v_{6},1)$. Now we found $pq$ distinct elements in $\Par(\Delta)$ and by \cite{Cassels:GeometryOfNumbers} or \cite[Lemma 2]{Sebo:HilbertBases} the total number of lattice points contained in $\Par(\Delta)$ equals to the absolute value of the matrix, whose rows consists of the vertices of $\Delta$ and added a column with ones. One easily verifies that the absolute value of this matrix is $pq$ as well. On the other hand we will see that all the lattice simplices $\Delta$ of our generalization of the theorem of White have $G(\Delta)$ cyclic.
\end{example}
\section{Bernoulli functions}
\label{sec:BernoulliFunctions}

In this section we will prove theorem (\ref{prop:BernoulliFunctions}). The proof will be based on \cite[Section 1]{MorrisonStevens:TerminalQuotSing}.

The case $n=2$ is obvious so in the following $n$ shall denote an integer bigger than $2$. We will write $G_n$ for the abelian group of units modulo $n$, \ie $G_n=\left(\IZ/n\IZ\right)^*$. Let $\IC[G_n]$ be the group algebra generated as a vector space over the complex numbers by the group isomorphic to $G$, the isomorphism given by $g\in G_n\mapsto \sigma_g$, \ie $\IC[G_n]=\{\sum_{g\in G_n}a_g \sigma_g:a_g\in\IC\}$. Let $g\in G_n$ and $x\in\tfrac{1}{n}\IZ/\IZ$. Let $k$ be and integer representative of $g$ modulo $n$, \ie $k\in\IZ$, $g\equiv k\pmod{n}$, and let $y$ be a rational number representing $x$ modulo $\IZ$, \ie $x\equiv y\pmod{\IZ}$. Then $ky$ is a rational number which is unique modulo $\IZ$. Since the \textit{1}st Bernoulli function $B_1$ is a periodic function with period $1$ the expression $B_1(ky)$ does not depend on the choice of the representatives. So writing $B_1(gx)$ for elements $g\in G_n$ and $x\in\tfrac{1}{n}\IZ/\IZ$ makes sence. To any $x\in\tfrac{1}{n}\IZ/\IZ$ we define its \emph{Stickelberger element} by

\[
 S(x):=\sum_{g\in G_n}B_1(gx)\sigma_g
\]

Let $U$ be the vector subspace of $\IC[G_n]$ generated by the image of $S$, \ie $U=\Span(S(x):x\in\tfrac{1}{n}\IZ/\IZ)$. Let $\{\sigma_g^*:g\in G_n\}$ be the basis of the dual vector space $\IC[G_n]^*$ which is dual to the natural basis $\{\sigma_g:g\in G_n\}$ of $\IC[G_n]$ and let $\left<\cdot,\cdot\right>:\IC[G_n]\times\IC[G_n]^*\rightarrow\IC;(v,f)\mapsto f(v)$ be the natural pairing of the dual space. The key idea is to show that a basis of the orthogonal complement $U^\bot=\{f\in \IC[G_n]^*:f(u)=0\text{ for all }u\in U\}$ of the subspace $U$ is given by all elements $u_g^*:=\sigma_g^*+\sigma_{-g}^*$ for $g\in G_n$.

\begin{theorem}\label{thm:BasisOfWBot}
 $\{u_g^*:g\in G_n\}$ is a basis of $U^\bot$. In particular $\dim_{\IC}(U^\bot)=\dim_{\IC}(U)=\varphi(n)/2$.
\end{theorem}

Let us grant this theorem for a moment

\begin{proof}[Proof of proposition (\ref{prop:BernoulliFunctions})]
 Let $\IZ\rightarrow G_n; k\mapsto \bar k:=k+n\IZ$ be the natural epimorphism. Let $u^*:=\sigma_{\overline{a_1}}^*+\sigma_{\overline{a_2}}^*+\ldots+\sigma_{\overline{a_n}}^*\in \IC[G_n]^*$. By hypothesis for all integers $t$
 
 \begin{multline*}
  \left<S\left(\frac{t}{n}\right),u^*\right>=\left<\sum_{g\in G_n}B_1\left(\frac{tg}{n}\right)\sigma_g,\sigma_{\overline{a_1}}^*+\sigma_{\overline{a_2}}^*+\ldots+\sigma_{\overline{a_n}}^*\right>\\
  =B_1\left(\frac{ta_1}{n}\right)+B_1\left(\frac{ta_2}{n}\right)+\ldots+B_1\left(\frac{ta_n}{n}\right)=0
 \end{multline*}
 
 Thus $u^*\in U^\bot$. Let $k_g$ for all elements $g$ of $G_n$ be positive integers such that $u^*=\sum_{g\in G_n}k_g \sigma_g^*$. For what follows we will consider the group $G_n$ as the set $\{i:i=1,2,\ldots,n;\gcd(i,n)=1\}$ together with the obvious composition. By theorem (\ref{thm:BasisOfWBot}) there exists a set of complex numbers $\{\lambda_g\in\IC:g\in G_n,g\le n/2\}$ such that
 
 \[
  u^*=\sum_{g\in G_n}k_g \sigma_g^*=\sum_{\substack{g\in G_n\\ g\le n/2}}\lambda_g(\sigma_g^*+\sigma_{-g}^*)
 \]
 
 Hence $k_g=k_{-g}$ for all $g\in G_n$.
\end{proof}

Next we will prove theorem \ref{thm:BasisOfWBot}.

\begin{proof}[Proof of \ref{thm:BasisOfWBot}]
 Let $\rho:G\rightarrow\mathrm{GL}(\IC[G_n])$ be the regular representation of $G_n$, \ie for all $g\in G_n$
 \[
  \rho_g:\IC[G_n]\rightarrow\IC[G_n];\sigma_h\mapsto \sigma_{gh}
 \]
 In other words we consider the regular module $\IC[G_n]^\circ$, \ie $\IC[G_n]$ becomes a $\IC[G_n]$--module by the scalar multiplication induced by the multiplication of the algebra. Since we work over the field of complex numbers there exists a decomposition of $\IC[G_n]$ into a direct sum $W_1\oplus W_2\oplus\ldots\oplus W_h=\IC[G_n]$ of irreducible $\IC[G_n]$--submodules of $\IC[G_n]$. Since $G_n$ is commutative it follows that all the irreducible $\IC[G_n]$--modules $W_i$ are $1$--dimensional as complex vector spaces and pairwise nonisomorphic. Let $1=\sum_{i=1}^{|G_n|}e_i$ with nonzero idempotents $e_i\in W_i$, \ie $e_i^2=e_i$, which satisfay $e_ie_j=0$ for $i\neq j$. Then $e_i$ is a (vector space) basis of $W_i$ for all $i=1,2,\ldots,|G_n|$ and all the $e_i$ together form a basis of $\IC[G_n]$. By \cite[Theorem 2.12]{Isaacs:CharacterTheory} we have that 
 
 \[
  e_i=\frac{1}{|G_n|}\sum_{g\in G_n}\chi_i(1)\chi_i(g^{-1})\sigma_g
 \]
 
 where $\chi_i$ is the corresponding character to the irreducible $\IC[G_n]$--module $W_i$. Since $G_n$ is commutative the characters $\chi_i$ have degree $1$, \ie $\chi_i(1)=1$, \ie $e_i=1/|G_n|\sum_{g\in G_n}\chi_i(g^{-1})\sigma_g$. The characters $\chi_i$ for $i=1,2,\ldots,\varphi(n)$ are all possible characters of degree $1$ of $G_n$ which we will denote by $C(n)$, \ie $C(n)=\{\chi_i:i=1,2,\ldots,\varphi(n)\}=\{\chi:G_n\rightarrow\IC^*\text{ homomorphism of groups}\}$. Let $\chi\in C(n)$. If $\chi(-1)=-1$ then we call the character \emph{odd}. We will denote the set of all odd characters of degree $1$ of $G_n$ by $C^-(n)$. It holds that $|C^-(n)|=\varphi(n)/2$. Let $m$ be a positive integer dividing $n$. By the natural morphism $i:G_n\rightarrow G_m;x+n\IZ\mapsto x+m\IZ$ we can pull back any character $\chi\in C(m)$ by $i^*\chi:=\chi\circ i\in C(n)$. For any character $\chi\in C(n)$ there exists a smallest unique positive integer $f_\chi\mid n$, called the \emph{conductor} of $\chi$, and a unique character $P\chi\in C(f_\chi)$ such that $\chi=j^*P\chi$ where $j$ is the natural morphism $j:G_n\rightarrow G_{f_\chi}$. Let $\IZ\rightarrow G_n;k\mapsto \bar k:=k+n\IZ$ be the natural epimorphism. From now on we will regard a character $\chi\in C(n)$ as a map $\IZ\rightarrow\IC$ which maps an integers $a$ to $P\chi(\bar a)$ if $(a,f_\chi)=1$ and $0$ otherwise and we will call it a \emph{Dirichlet character}. Then the function $\chi:\IZ\rightarrow\IC$ satisfays the following conditions
 
 \begin{enumerate}
  \item $\chi(a)=\chi(b)$ if $a\equiv b\pmod{f_\chi}$
  \item $\chi(ab)=\chi(a)\chi(b)$ for all $a,b\in\IZ$
  \item $\chi(a)=0$ if $(a,f_\chi)\neq1$
  \item $\chi(a)=P\chi(a)$ if $(a,f_\chi)=1$
 \end{enumerate}

 By (1) and since $f_\chi\mid n$ it still makes sense to evaluate the function $\chi:\IZ\rightarrow\IC$ on elements $g$ of $G_n$: just choose any integer $k$ such that $k\equiv g\pmod{n}$. To any character $\chi\in C(n)$ we will assign a complex number

 \[
  B_{1,\chi}=\sum_{k=1}^{f_\chi}\chi(k)B_1\left(\frac{k}{f_\chi}\right)
 \]

 We will need the following nontrivial result on the nonvanishing of $B_{1,\chi}$. Indeed Washington writes in \cite[p. 38]{Washington:CyclotomicFields}: \enquote{\dots Note that the theorem implies that $B_{1,\chi}\neq0$ if $\chi$ is odd. There is no elementary proof known for this fact\dots}. A proof can be found in \cite[\S2, Theorem 2]{Iwasawa:PAdicLFunc} or in \cite[Chapter 4]{Washington:CyclotomicFields}.

 \begin{theorem}\label{thm:IwasawaNonVanish}
  If $\chi$ is an odd character, then $B_{1,\chi}\neq0$.
 \end{theorem}
 
 For all odd Dirichlet characters $\chi_i$ where $\chi_i$ is the character corresponding to the irreducible $\IC[G_n]$--module $W_i$ (see above) we have by theorem \ref{thm:IwasawaNonVanish} that $B_{1,\chi_i}\neq0$. We rescale the vectors $e_i$ by the factor $|G_n|B_{1,\chi_i}$
 \begin{eqnarray*}
  u_{\chi_i}:=|G_n|B_{1,\chi_i}e_i&=&B_{1,\chi_i}\sum_{g\in G_n}\chi_i(g^{-1})\sigma_g=\sum_{g\in G_n}B_{1,\chi_i}\chi_i(g^{-1})\sigma_g=\\
  &=&\sum_{g\in G_n}\sum_{k=1}^{f_{\chi_i}}\underbrace{\chi_i(k)\chi_i(g^{-1})}_{\chi_i(kg^{-1})}B_1\left(\frac{k}{f_{\chi_i}}\right)\sigma_g\\
  &\stackrel{k':=kg^{-1}}{=}&\sum_{g\in G_n}\sum_{k'=1}^{f_{\chi_i}}\chi_i(k')B_1\left(\frac{k'g}{f_{\chi_i}}\right)\sigma_g\\
  &=&\sum_{k'=1}^{f_{\chi_i}}\chi_i(k')\sum_{g\in G_n}B_1\left(\frac{k'g}{f_{\chi_i}}\right)\sigma_g=\sum_{k'=1}^{f_{\chi_i}}\chi_i(k')S\left(\frac{k'}{f_{\chi_i}}\right)
 \end{eqnarray*}
 Thus for all $\chi\in C^-(n)$ it follows that the vector $u_\chi$ is contained in $U$. Since the $e_i$ are linearly independent we get that $\{u_\chi:\chi\in C^-(n)\}$ is a linearly independent set. Hence $\dim U\ge\varphi(n)/2$. Since $\dim U+\dim U^\bot=\dim\IC[G_n]=\varphi(n)$ it follows that $\dim U^\bot\le\varphi(n)/2$. Obviously $u_g^*$ is contained in $U^\bot$ since for all $x\in\tfrac{1}{n}\IZ/\IZ$
 \[
  \left<S(x),u_g^*\right>=\left<\sum_{g\in G_n} B_1(gx)g,\sigma_g^*+\sigma_{-g}^*\right>=B_1(gx)+B_1(-gx)=0
 \]
 
 Furthermore since $\{\sigma_g^*: g\in G_n\}$ form a basis of $\IC[G_n]^*$ we get that $\{u_g^*:g\in G_n\}$ is linearly independent. Thus $\dim U^\bot\ge\varphi(n)/2$.
\end{proof} 
\section{The proof}
\label{sec:TheProof}

In this section we are going to prove theorem \ref{thm:GeneralizedWhite} and theorem \ref{thm: Cyclic Quotient Singularities}. The following technical assertion will be needed in the proof of theorem \ref{thm:GeneralizedWhite}.

\begin{proposition}\label{prop:TechnicalAssertionForGenWhiteProof}
 Let $\Delta=\conv(v_1,\ldots,v_{2d})\subset\R^{2d-1}$ be a $(2d-1)$--dimensional lattice simplex whose codimension 1 facets are basic simplices such that for all $k=1,\ldots,d-1$, $\mathrm{Int}(k\Delta)\cap\Z^{2d-1}=\leer$, \ie $h_\Delta=1+(n-1)t^d$ for an integer $n>0$, \ie for all $k=1,\ldots,d-1$, $k\Delta\cap\Z^{2d-1}\subset\Z v_1+\ldots+\Z v_{2d}$ (see proposition \ref{prop: h-Polynomial vs lattice free simplex}). Let $\sum_{i=1}^{2d}\lambda_i(v_i,1)\in\Z^{2d}$ be a lattice vector for real numbers $0\le\lambda_i<1$. Then either $\sum_i\lambda_i=0$ or $\sum_i\lambda_i=d$. Furthermore, if $\sum_i\lambda_i=d$ then $\lambda_i\neq0$ for all $i$.
\end{proposition}
\begin{proof}
 Let $0\neq\sum_{i=1}^{2d}\lambda_i(v_i,1)\in\Z^{2d}$ for real numbers $0\le\lambda_i<1$. Then $k:=\sum_{i=1}^{2d}\lambda_i$ is an integer. Hence $\sum_{i=1}^{2d}\lambda_iv_i$ is a lattice vector of $\Z^{2d-1}$ which is contained in the multiple $k\Delta$ and which is not an integer linear combination of the vertices of $\Delta$. So by assumption $k\ge d$. We claim that $k\le d$ as well: Assume by contradiction that $k>d$. Then $\sum_{i=1}^{2d}\fractional{1-\lambda_i}v_i$ is a lattice vector of $\Z^{2d-1}$ which is not an integer linear combination of the vertices of $\Delta$. Furthermore we have that $\sum_{i=1}^{2d}\fractional{1-\lambda_i}\le 2d-\sum_{i=1}^{2d}\lambda_i<d$ which contradicts our assumption. So $\sum_{i=1}^{2d}\lambda_i=d$.
  
  If $0=\sum_{i=1}^{2d}\lambda_i(v_i,1)\in\Z^{2d}$ is a lattice vector for real numbers $0\le\lambda_i<1$ then, since $(v_1,1),\ldots,(v_{2d},1)$ are linearly independent, $\lambda_i=0$ for all i and thus $\sum_{i=1}^{2d}\lambda_i=0$.

  Now assume that $0\neq\sum_{i=1}^{2d}\lambda_i(v_i,1)\in\Z^{2d}$ like above but with \eg $\lambda_1=0$. Then $0\neq\sum_{i=1}^{2d}\fractional{1-\lambda_i}(v_i,1)\in\Z^{2d}$ with $0\le\fractional{1-\lambda_i}<1$ and $\sum_{i=1}^{2d}\fractional{1-\lambda_i}\le 2d-1-\sum_{i=1}^{2d}\lambda_i=d-1$. Contradiction.
\end{proof}

\begin{proof}[Proof of Theorem \ref{thm:GeneralizedWhite}]
 Let $\Delta=\conv(v_1,\ldots,v_{2d})\subset\R^{2d-1}$ be a $(2d-1)$--dimensional lattice simplex whose codimension $1$ facets are basic lattice simplices such that for all $k=1,\ldots,d-1$, $\mathrm{Int}(k\Delta)\cap\Z^{2d-1}=\leer$. We consider the lattice simplex
\[
 \Delta':=\conv((v_1,1),\ldots,(v_{2d},1))\subset\{(x_1,\ldots,x_{2d})\in\R^{2d}:x_{2d}=1\}\subset\R^{2d}
\]
Since by assumption $\Gamma:=\conv(v_1,\ldots,v_{2d-1})$ is a basic lattice simplex, there exists a unimodular linear isomorphism $\varphi:\Z^{2d}\rightarrow\Z^{2d}$ such that $\varphi(\Gamma)=\conv(e_{2d},e_1+e_{2d},\ldots,e_{2d-2}+e_{2d})$ and $\varphi(v_{2d},1)=(v_{2d}',1)$ for an integer vector $v_{2d}'\in\Z^{2d-1}$. Thus
\[
 G(\Delta):=\Z^{2d}/\left(\Z(v_1,1)+\ldots+\Z(v_{2d},1)\right)\cong\Z/N\Z
\]
for an integer $N>0$. Take a generator $(w,m)=\sum_{i=1}^{2d}\lambda_i(v_i,1)$ of $G(\Delta)$ with rational $0\le\lambda_i<1$. Here we used the fact that any integer vector $v\in\Z^{2d}$ can be uniquely represented by an element of $\Z(v_1,1)+\ldots+\Z(v_{2d},1)$ and an integer vector of $\{\sum_{i=1}^{2d}\mu_i(v_i,1):0\le\mu_i<1\}$. We can write
\[
 \lambda_i=\frac{a_i}{n}\text{ for integers }a_i,n\text{ with }n>0\text{ and }(a_1,\ldots,a_{2d},n)=1
\]
The last condition ensures that for all $t\in\Z\ohne n\Z$
\[
 0\neq\sum_{i=1}^{2d}\fractional{\frac{ta_i}{n}}(v_i,1)\in\Par(\Delta)
\]
Since $(w,m)$ is a generator of $G(\Delta)\cong\Z/N\Z$ this shows $N=n$. By proposition \ref{prop:TechnicalAssertionForGenWhiteProof} it follows that for all $t\in\Z\ohne n\Z$ and for all $i=1,\ldots,2d$, $\{ta_i/n\}\neq0$, \ie $(a_i,n)=1$ for all $i=1,\ldots,2d$.  This shows that we can assume that $a_{2d}=1$; just take $\sum_{i=1}^{2d}\{ta_i/n\}(v_i,1)$ for an appropriate $t\in\Z$ with $(t,n)=1$. Then $(v_{2d},1)=n(w,m)-\sum_{i=1}^{2d-1}a_i(v_i,1)$, \ie $\Z(v_1,1)+\ldots+\Z(v_{2d},1)\subset\Z(v_1,1)+\ldots+\Z(v_{2d-1},1)+\Z(w,m)$. Hence
\[
\Z^{2d}/\left(\Z(v_1,1)+\ldots+\Z(v_{2d-1},1)+\Z(w,m)\right)=0
\]
since $(w,m)$ is a generator of $G(\Delta)$. In other words $\{(v_1,1),\ldots,(v_{2d-1},1),(w,m)\}$ is a lattice basis for $\Z^{2d}$.

Obviously, since $(a_i,n)=1$ for all $i=1,\ldots,2d$, it holds for all $t\in\Z\ohne n\Z$
\[
 d=\sum_{i=1}^{2d}\fractional{\frac{ta_i}{n}}\text{ (by proposition \ref{prop:TechnicalAssertionForGenWhiteProof})}\Leftrightarrow 0=\sum_{i=1}^{2d}B_1\left(\frac{ta_i}{n}\right)
\]
The right hand equation is trivially satisfied for all integers $t\in n\Z$ since we have set $B_1(k)=0$ for all integers $k$. Then by proposition \ref{prop:BernoulliFunctions} we can assume (after reordering the vectors $v_1,\ldots,v_{2d-1}$) that $a_{2i-1}+a_{2i}\equiv0\pmod{n}$ for all $i=1,,\ldots,d$. The vector $v_{2d}$ we can leave at its place.

Let $\varphi$ be the unique unimodular linear isomorphism which maps the basis
\[
 \{(v_1,1),\ldots,(v_{2d-1},1),(w,m)-(v_2,1)-(v_4,1)-\ldots-(v_{2d-2},1)\}
\]
of $\Z^{2d}$ on the basis
\[
 \{e_1\times0,e_1\times e_1,e_2\times0,e_2\times e_2,\ldots,e_d\times0,e_d\times e_d\}
\]
Then
\begin{align*}
 \varphi(v_{2d},1)&=\varphi(n(w,m)-\sum_{i=1}^{2d-1}a_i(v_i,1))=n\varphi((w,m)-\sum_{i=1}^{d-1}(v_{2i},1))-\\
 &-\sum_{i=1}^{d-1}\left(a_{2i-1}\varphi(v_{2i-1},1)+(a_{2i}-n)\varphi(v_{2i},1)\right)-a_{2d-1}\varphi(v_{2d-1},1)
\end{align*}
Since $a_{2i-1}+a_{2i}=n$ for all $i=1,\ldots,d$, it follows
\begin{align*}
 \varphi(v_{2d},1)&=ne_d\times e_d-\sum_{i=1}^{d-1}\left(a_{2i-1} e_i\times0-a_{2i-1}e_i\times e_i\right)-(n-1)e_d\times0\\
 &=e_d\times(a_1,a_3,a_5,\ldots,a_{2d-3},n)
\end{align*}
Hence $\Delta$ is affine unimodular isomorphic to $\Delta_1*\ldots*\Delta_d$ for
\begin{align*}
 \Delta_i&:=\conv(0,e_i)\subset\R^d\text{ for }i=1,\ldots,d-1\text{ and }\\
 \Delta_d&:=\conv(0,(a_1,a_3,a_5,\ldots,a_{2d-3},n))\subset\R^d
\end{align*}
 This proves (1)$\Rightarrow$(2).

 Now let $\Delta\subset\R^{2d-1}$ be a $(2d-1)$--dimensional lattice simplex whose codimension $1$ facets are basic lattice simplices such that $\Delta\cong\Delta_1*\ldots*\Delta_d$ for $1$--dimensional lattice simplices $\Delta_i=\conv(u_{i1},u_{i2})\subset\R^d$. Let $v_1,\ldots,v_{2d}\subset\R^{2d-1}$ be the vertices of $\Delta$ and $\sum_{i=1}^{2d}\lambda_iv_i\in\Z^{2d-1}$ for rational numbers $0<\lambda_i$ such that $\sum_i\lambda_i=k\in\{1,\ldots,d-1\}$, \ie $\sum_i\lambda_iv_i\in\mathrm{Int}(k\Delta)\cap\Z^{2d-1}$.

We consider $\Delta':=\conv((v_1,1),\ldots,(v_{2d},1))\subset\Z^{2d}$. Then $\sum_{i=1}^{2d}\lambda_i(v_i,1)\in\Z^{2d}$. Furthermore we observe that there is a unimodular linear isomorphism which maps $\Delta'$ on $\Delta_1*\ldots*\Delta_d$. Now the following observation is important. If an element $\sum_{i=1}^d\left(\mu_{i1}e_i\times u_{i1}+\mu_{i2}e_i\times u_{i2}\right)\in\Z^{2d}$ for rational numbers $\mu_{ij}$, then $\mu_{i1}+\mu_{i2}\in\Z$. Hence by reordering the $\lambda_i$ we may assume that $\lambda_{2i-1}+\lambda_{2i}\in\Z$ for $i=1,\ldots,d$. Thus $\lambda_{2i-1}+\lambda_{2i}\ge1$ for $i=1,\ldots,d$, \ie $\sum_i\lambda_i\ge d$.
\end{proof}

\begin{proof}[Proof of Theorem \ref{thm: Cyclic Quotient Singularities}]
 The direction (2)$\Rightarrow$(1) easily follows by the general fact that $\{x\}+\{-x\}=1$ for all real numbers $x$ not being an integer.

 Next assume that $(a_1/n,\ldots,a_{2d}/n)$ is an isolated cyclic quotient singularity with minimal $\log$--discrepancy$\ge d$, \ie $\gcd(a_i,n)=1$ and for all $t=1,\ldots,n-1$, $\sum_{i=1}^{2d}\{ta_i/n\}\ge d$. Then as above it follows that $\sum_{i=1}^{2d}\{ta_i/n\}=d$. Indeed assume by contradiction that for a $t=1,\ldots,n-1$ it is $\sum_{i=1}^{2d}\{ta_i/n\}>d$. Then
 \[
  \sum_{i=1}^{2d}\fractional{\frac{(n-t)a_i}{n}}=\sum_{i=1}^{2d}\fractional{\frac{-ta_i}{n}}=2d-\sum_{i=1}^{2d}\fractional{\frac{ta_i}{n}}<d
 \]
 contradicting the fact that the minimal $\log$--discrepancy$\ge d$.
 Then the assertion follows by proposition \ref{prop:BernoulliFunctions}.
\end{proof}


\bibliographystyle{amsalpha} 
\bibliography{biblio}

\end{document}